\documentclass[11pt,reqno]{amsart}

\usepackage{color}
\usepackage[colorlinks=true, allcolors=blue]{hyperref}
\usepackage{amsmath, amssymb, amsthm}
\usepackage{mathrsfs}
\usepackage{mathtools}
\usepackage[noabbrev,capitalize,nameinlink]{cleveref}
\crefname{equation}{}{}
\usepackage{fullpage}
\usepackage[noadjust]{cite}
\usepackage{graphics}
\usepackage{pifont}
\usepackage{tikz}
\usepackage{bbm}
\usepackage[T1]{fontenc}

\usetikzlibrary{arrows.meta}

\usepackage{environ}
\usepackage{framed}
\usepackage{url}
\usepackage[linesnumbered,ruled,vlined]{algorithm2e}
\usepackage[noend]{algpseudocode}
\usepackage[labelfont=bf]{caption}
\usepackage{cite}
\usepackage{framed}
\usepackage[framemethod=tikz]{mdframed}
\usepackage{appendix}
\usepackage{graphicx}
\usepackage[textsize=tiny]{todonotes}
\usepackage{tcolorbox}
\usepackage{enumerate}
\allowdisplaybreaks[1]
\usepackage{enumerate}

\crefname{algocf}{Algorithm}{Algorithms}

\crefname{equation}{}{} 
\AtBeginEnvironment{appendices}{\crefalias{section}{appendix}} 

\usepackage[color,final]{showkeys} 

\colorlet{refkey}{orange!20}
\colorlet{labelkey}{blue!30}

\crefname{algocf}{Algorithm}{Algorithms}

\numberwithin{equation}{section}
\newtheorem{theorem}{Theorem}[section]
\newtheorem{proposition}[theorem]{Proposition}
\newtheorem{lemma}[theorem]{Lemma}

\crefname{claim}{Claim}{Claims}

\newtheorem{conjecture}[theorem]{Conjecture}
\newtheorem*{question*}{Question}

\theoremstyle{definition}
\newtheorem{definition}[theorem]{Definition}

\newtheorem*{definition*}{Definition}

\theoremstyle{remark}
\newtheorem*{remark}{Remark}

\newcommand{\snorm}[1]{\lVert#1\rVert}

\newcommand{\mb}{\mathbb}

\newcommand{\mc}{\mathcal}

\newcommand{\on}{\operatorname}

\allowdisplaybreaks
\title{Rank deficiency of random matrices}

\author[A1]{Vishesh Jain}
\address{Department of Statistics, Stanford University, Stanford CA 94305, USA}
\email{visheshj@stanford.edu}

\author[A2]{Ashwin Sah}
\author[A3]{Mehtaab Sawhney}
\address{Department of Mathematics, Massachusetts Institute of Technology, Cambridge, MA 02139, USA}
\email{\{asah,msawhney\}@mit.edu}

\begin{document}

\begin{abstract}
Let $M_n$ be a random $n\times n$ matrix with i.i.d.~$\on{Bernoulli}(1/2)$ entries. We show that for fixed $k\ge 1$, 
\[\lim_{n\to \infty}\frac{1}{n}\log_2\mb{P}[\on{corank}M_n\ge k] = -k.\]
\end{abstract}

\maketitle

\section{Introduction}\label{sec:introduction}

A fundamental, and intensely studied, problem in combinatorial random matrix theory is the determination of the probability of singularity of $n\times n$ random Bernoulli matrices (i.e.~ $n\times n$ matrices for which each entry is independently $0$ or $1$ with equal probability). The study of this problem was initiated in work of Koml\'os \cite{Kom67}. After intermediate works over a period of over 50 years \cite{KKS95, TV06, TV07, BVW10}, the breakthrough work of Tikhomirov \cite{Tik20} showed that for any fixed $p \in (0,1/2]$,
\[\mb{P}[M_{n}(\on{Ber}(p)) \text{ is singular}] = (1-p + o_n(1))^{n},\]
where we use the notation $M_{m\times n}(\xi)$ to denote an $m\times n$ random matrix with i.i.d.~entries distributed as $\xi$, and the lighter notation  $M_n(\xi)$ for $M_{n\times n}(\xi)$. Also, $\on{Ber}(p)$ is the random variable which takes on the value $1$ with probability $p$ and $0$ with probability $1-p$.

By considering the probability of a row of the matrix being $0$, one sees that the result of Tikhomirov is optimal up to the $o_n(1)$ term. Recently, several works \cite{Hua20, JSS20discrete1, JSS20discrete2, LT20, BR18} have addressed the more refined question of determining the probability of singularity of $M_{n}(\xi)$ up to a $(1+o_n(1))$ factor; in contrast, the aforementioned result of Tikhomirov determines this probability only up to a subexponential (in $n$) factor. While these works have succeeded in the case of sparse Bernoulli matrices (with sparsity allowed to depend on $n$) \cite{Hua20, LT20, BR18},  as well as in the case of a fixed $\xi$ which is not uniform on its support \cite{JSS20discrete1, JSS20discrete2}, we note that for the case of $\on{Ber}(1/2)$, the estimate of Tikhomirov remains essentially the best known.

An equivalent condition to singularity of $M_{n}(\xi)$ is that the corank of $M_{n}(\xi)$ is at least $1$. Given this formulation, the following question immediately suggests itself: given $k\ge 1$, what is the probability that the corank of $M_{n}(\xi)$ is at least $k$? In the case when $\xi = \on{Ber}(p)$ for fixed $p \in (0,1/2]$, by considering the event that first $k$ rows of the matrix are identically $0$, we see that this probability is at least $(1-p)^{nk}$. The previously best-known upper bound appears to be due to Kahn, Koml\'os, and Szemer\'edi \cite{KKS95}, who showed that there exists a function $f \colon \mb{N} \to \mb{R}^+$ with $f(k) \to 0$ as $k \to \infty$ such that 
\[\mb{P}[\on{corank}M_{n}(\on{Ber}(1/2))\ge k] \le f(k)^{n}.\]
The above simple lower bound shows that the decay of $f(k)$ can be at most $2^{-k}$; it has been suggested (cf.~\cite[Section~4]{Vu20}, `It is tempting to conjecture...') that this rate of decay is essentially sharp, i.e.~that 
\[\mb{P}[\on{corank}M_{n}(\on{Ber}(1/2))\ge k] = (1/2 + o_n(1))^{kn}.\]

The main result of this paper confirms this belief.

\begin{theorem}\label{thm:main}
Fix $p\in (0,1/2]$ and let $\xi = \on{Ber}(p)$. Fix $k\ge 1$ and $\epsilon > 0$. Then, for $n\ge n_{\ref{thm:main}}(p,k,\epsilon)$, we have
\[\mb{P}[\on{corank}M_n(\xi)\ge k]= (1-p+\epsilon)^{kn}.\]
\end{theorem}
\begin{remark}
A modification of our proof, with \cref{prop:structure} replaced by the corresponding versions in \cite{JSS20discrete1,JSS20discrete2}, shows that for any fixed $\xi$ which is supported on finitely many points, 
\[\mb{P}[\on{corank}M_n(\xi)\ge k]\le(\max_{z\in\mb{R}}\mb{P}[\xi = z]+o_n(1))^{kn}.\]
\end{remark}
We conjecture that, in general, the following holds.
\begin{conjecture}\label{conj:general}
Fix a random variable $\xi$ supported on finitely many points. Fix $k\ge 1$ and $\epsilon > 0$. Then for $n\ge n_{\ref{conj:general}}(\xi,k,\epsilon)$, we have
\[\mb{P}[\on{corank}M_n(\xi)\ge k]\le(\mb{P}[\xi = 0]+\epsilon)^{kn}+(\mb{P}[\xi_1=\cdots=\xi_{k+1}]+\epsilon)^n,\]
where $\xi_1,\ldots,\xi_{k+1}$ are independent samples of $\xi$.
\end{conjecture}
A stronger conjecture is that the dominant contribution to the probability of the corank being at least $k$ comes from the event of having $k-t+1$ rows equal up to sign and $t$ rows zero for some $0\le t\le k$, or the same for columns (i.e., this controls the probability up to a $(1+o_n(1))$ factor). All such events contribute, for instance, when $\xi = \on{Ber}(1/2)$. However, in the case when $\xi = \on{Ber}(p)$, $p \in (0,1/2)$, the event of having $k$ rows which are identically zero is exponentially more likely than any of the other events. In concurrent and independent work, Huang \cite{Hua21} has proved this stronger conjecture (with accompanying singular value bounds) for sufficiently sparse Bernoulli matrices i.e.~$\xi = \on{Ber}(p_n)$ with $$1 \le \liminf_{n \to \infty}\frac{p_n \cdot n}{\log{n}} \le \limsup_{n \to \infty}\frac{p_n \cdot n}{\log{n}} < \infty;$$ 
it is plausible that, combined with the techniques in \cite{LT20}, the upper bound on the $\limsup$ can be relaxed (perhaps even up to $p_n \le c$ for some small constant $c$). In the complementary dense case considered here, we leave the resolution of this stronger conjecture as a subject for future research.

Finally, we mention that in recent years, there have been several other works on the (co)rank of random matrices (cf.~\cite{CV08, CEGHR20}). However, the focus of these works is on the asymptotic determination of the (co)rank of various models, which is different from our focus on the determination of (sharp) non-asymptotic rates to have corank at least $k$ for matrices which are of full rank with high probability.   


\subsection{Notation}\label{sub:notation}
Given a positive integer $N\ge 1$, let $\mb{S}^{N-1}$ be the set of unit vectors in $\mb{R}^{N}$. Let $\snorm{\cdot}_2$ be the Euclidean norm. For a matrix $A = (A_{ij})$, let $\snorm{A}$ be its spectral norm (i.e., $\ell^2\to\ell^2$ operator norm) and let $\snorm{A}_{\on{HS}}$ be its Hilbert-Schmidt norm, defined as
\[\snorm{A}_{\on{HS}}^2 = \sum A_{ij}^2.\]
We let $[N]$ denote the discrete interval $\{1,\dots,N\}$. Given an $m\times n$ matrix $A$ and a subset $S \subseteq [n]$ of columns, we let $A_S$ be the $m\times |S|$ submatrix of $A$ consisting only of the columns in $S$. 

For an $\mb{R}^{N}$-valued random variable $\xi$ and a real number $r \ge 0$, we define the L\'evy concentration function by 
\[\mc{L}(\xi, r):= \sup_{z \in \mb{R}^N}\mb{P}[\snorm{\xi-z}_2\le r].\]
Note that the case $N = 1$ coincides with the usual (scalar) L\'evy concentration function. We let $\ell_1(\mb{Z})$ denote the set of functions $f\colon\mb{Z}\to\mb{R}$ satisfying $\sum_{z\in\mb{Z}}|f(z)| < \infty$.

We also make use of asymptotic notation. Given functions $f,g$, we write $f = O_{\alpha}(g)$ or $f\lesssim_{\alpha} g$ to mean $f \le C_\alpha g$, where $C_\alpha$ is some constant depending on $\alpha$. We write $f = \Omega_\alpha(g)$ or $f \gtrsim_\alpha g$ to mean $f \ge c_{\alpha} g$, where $c_\alpha > 0$ is some constant depending on $\alpha$. Finally, we write $f = \Theta_\alpha(g)$ to mean that both $f = O_\alpha(g)$ and $f = \Omega_\alpha(g)$ hold. For parameters $\epsilon, \delta$, we write $\epsilon \ll_{\alpha} \delta$ to mean that $\epsilon\le c_{\alpha}(\delta)$ for a sufficiently decaying function $c_\alpha$ depending on $\alpha$. 

Finally, we will omit floors and ceilings when they make no essential difference.

\section{Proof of \texorpdfstring{\cref{thm:main}}{Theorem 1.2}}\label{sec:proof}

\subsection{Preliminaries}
\label{sub:prelims}

We collect some (by now) standard notions in the non-asymptotic theory of random matrices. For parameters $\delta, \rho \in (0,1)$ and an integer $n \ge 1$, $\on{Comp}_{n}(\delta, \rho)$ denotes the set of unit vectors in $\mb{R}^n$ which have Euclidean distance at most $\rho$ to the set of $\delta n$-sparse vectors. $\on{Incomp}_{n}(\delta, \rho) := \mb{S}^{n-1} \setminus \on{Comp}_{n}(\delta, \rho)$. When the ambient dimension is clear from context, we will drop the subscript $n$. 

We will need to consider the anticoncentration behavior of a vector with respect to i.i.d.~$\on{Ber}(p)$ random variables. For this, we will use the threshold function, which was isolated in the work of Tikhomirov \cite{Tik20}. 

\begin{definition}\label{def:threshold}
For $p\in(0,1/2]$, $L \ge 1$, and $x\in\mb{R}^n$, we define
\[\mc{T}_p(x,L) := \sup\bigg\{t\in(0,1): \mc{L}\bigg(\sum_{i=1}^nb_ix_i, t\bigg) > Lt\bigg\},\]
where $b_1,\dots, b_n$ are independent $\on{Ber}(p)$ random variables.
\end{definition}

\subsection{Overview of the proof}
In this subsection, we present the (short) proof of \cref{thm:main}, modulo the key \cref{prop:compressible,prop:structure}, which we will formally state and prove in the subsequent subsections. 

\begin{proof}[Proof of \cref{thm:main}]
For $\delta, \rho \in (0,1)$, let $\mc{E}_{C}(\delta, \rho)$ be the event that for every $(n-k)\times n$ sub-matrix $A$ of $M_{n}(\xi)$, and for all $x \in \mb{S}^{n-1}$ such that $Ax  = 0$, we have that $x \in \on{Comp}(\delta, \rho)$. In words, the right-kernel unit vectors of every $(n-k)\times n$ sub-matrix of $M_{n}(\xi)$ are in $\on{Comp}(\delta, \rho)$.  

Denote the rows of $M_n(\xi)$ by $R_1,\dots, R_n$. If $\on{rank}M_n(\xi)\le n-k$, then there must be some $k$ rows of $M_{n}(\xi)$ which are in the span of the remaining $n-k$ rows. Since the rows are i.i.d, it follows from the union bound that
\begin{align*}
\mb{P}[\on{rank}M_n(\xi)\le n-k]&\le\mb{P}[\mc{E}_C(\delta, \rho)] + \mb{P}[\on{rank} M_n(\xi)\le n-k\wedge\mc{E}_C(\delta, \rho)^c]\\
&\le \mb{P}[\mc{E}_C(\delta, \rho)] + \binom{n}{k}\mb{P}[\{R_1,\ldots,R_k\in\on{span}(R_{k+1},\ldots,R_n)\}\wedge\mc{E}_C(\delta, \rho)^c].
\end{align*}

In \cref{prop:compressible}, which is the key innovation of this work, we will show that there exist $\delta, \rho \in (0,1)$ (depending on $k, p, \epsilon$) such that
\[\mb{P}[\mc{E}_C(\delta, \rho)]\le (1-p+\epsilon)^{kn}.\]
Now, fix this choice of $\delta, \rho$ and denote the corresponding event $\mc{E}_C(\delta, \rho)$ simply by $\mc{E}_C$. Since $\binom{n}{k} \le n^{k} \le (1+\epsilon)^{n}$ for $n$ sufficiently large, it remains to show that
\[\mb{P}[\{R_1,\ldots,R_k\in\on{span}(R_{k+1},\ldots,R_n)\}\wedge\mc{E}_C^c] \le (1-p+\epsilon)^{nk}.\]

In \cref{prop:structure}, we show the following dichotomy: consider the $(n-k)\times n$ matrix $M_{(n-k)\times n}(\xi)$ formed by the rows $R_{k+1},\ldots, R_{n}$. With probability at least $1-4^{-kn}$, either 
\begin{itemize}
\item every unit vector in the right-kernel of $M_{(n-k)\times n}(\xi)$ is in $\on{Comp}(\delta,\rho)$, or 
\item there is a unit vector $v = v (R_{k+1},\ldots, R_n)$ in the right-kernel of $M_{(n-k)\times n}(\xi)$ with
\[\mc{T}_p(v,L_{\ref{prop:structure}})\le (1-p+\epsilon)^n,\]
where $L_{\ref{prop:structure}}$ is a constant depending on $k, p, \epsilon$.
\end{itemize}
Note that, on the event $\mc{E}_C^c$, the first case cannot occur. Let $\mc{A}$ denote the set of possible realizations of $R_{k+1},\ldots,R_n$ for which the second case occurs; for every such realisation $a \in \mc{A}$, we have a unit vector $v = v(a)$ satisfying the conclusion of the second case. Then, 
\begin{align*}
\mb{P}[R_1,\ldots,R_k\in\on{span}(R_{k+1},\ldots,R_n)\wedge\mc{E}_C^c]&\le\sup_{a\in\mc{A}}\mb{P}[\langle R_1, v(a)\rangle = \cdots = \langle R_{k},v(a)\rangle = 0] + 4^{-kn}\\
&\le L_{\ref{prop:structure}}^{k}(1-p+\epsilon)^{kn} + 4^{-kn}.
\end{align*}
The last inequality uses the independence of the rows, the definition of the threshold function, and the property $\mc{T}_p(v(a), L_{\ref{prop:structure}}) \le (1-p+\epsilon)^{n}$. The result now follows upon rescaling $\epsilon$.
\end{proof}

\subsection{Compressible vectors}\label{sub:compressible}
In this subsection, we prove \cref{prop:compressible}. Recall the event $\mc{E}_{C}(\delta, \rho)$ defined at the start of the proof of \cref{thm:main}.
\begin{proposition}\label{prop:compressible}
Fix $p\in(0,1/2]$, $k\ge 1$, and $\epsilon > 0$. There exist $\delta,\rho \in (0,1)$ (depending on $p, k, \epsilon$) such that
\[\mb{P}[\mc{E}_C(\delta,\rho)]\le (1-p+\epsilon)^{kn}.\]
\end{proposition}

The proof of \cref{prop:compressible} requires two ingredients, the first of which is the classical Kolmogorov-L\'evy-Rogozin anticoncentration inequality.
\begin{lemma}[\cite{Rog61}]
\label{lem:LKR}
Let $\xi_1,\dots,\xi_n$ be independent random variables. Then, for any real numbers $r_1,\dots,r_n > 0$ and any real number $r\ge\max_{i\in[n]}r_i$, 
\begin{align*}
    \mc{L}\bigg(\sum_{i=1}^{n}\xi_i, r\bigg) \le \frac{C_{\ref{lem:LKR}}r}{\sqrt{\sum_{i=1}^{n}(1-\mc{L}(\xi_i, r_i))r_i^2}},
\end{align*}
where $C_{\ref{lem:LKR}}>0$ is an absolute constant. 
\end{lemma}

The second ingredient is a version of restricted invertibility which allows one to select a subset of columns of full rank. 

\begin{lemma}[{\cite[Theorem~1]{GO11}}]\label{lem:restricted-invertibility}
Let $U$ be an $n\times m$ matrix of rank $n$. Then, there exists a subset $\mc{S} \subseteq [m]$ of columns of size $|\mc{S}| = n$ such that 
\[\snorm{U_\mc{S}^{-1}}_{\on{HS}}^2\le (m-n+1)\cdot \on{Tr}[(UU^{T})^{-1}].\]
\end{lemma}

The following is the key lemma in the proof of \cref{prop:compressible}.

\begin{lemma}\label{lem:levy-estimate}
Fix $p\in (0,1/2]$ and $k\ge 1$. There exists $\theta = \theta_{\ref{lem:levy-estimate}}(p,k) > 0$ for which the following holds. If $M$ is a $k\times n$ matrix whose rows are orthonormal vectors and $x$ is an $n$-dimensional random vector with independent $\on{Ber}(p)$ components, then
\[\mc{L}(Mx, \theta)\le (1-p)^k.\]
\end{lemma}
\begin{remark}
Our proof shows that $\theta(p,k)$ can be taken to be of size $\on{poly}(p)\exp(-\Omega(k))$. We suspect this is far from the truth and conjecture that one can take $\theta(p,k)$ to be of size $\on{poly}(p,1/k)$.
\end{remark}
\begin{proof}
For each $1\le i\le k$, let $\mc{T}_i\subseteq[n]$ be the indices $j \in [n]$ corresponding to the $\lfloor C_{\ref{lem:LKR}}^2 25^k/p\rfloor$ (which is $\ge k$) largest values $|M_{ij}|$, where (for concreteness) we break ties according to the natural ordering of the integers. Let $\mc{T} = \cup_{i=1}^k\mc{T}_i$. Since $\sum_{j=1}^{n}|M_{ij}|^{2} = 1$ for every $i \in [k]$, it follows that for every $j \in \mc{T}^{c}$,
\[|M_{ij}| \le 5^{-k}\sqrt{p}/C_{\ref{lem:LKR}} \quad \forall  i \in [k].\]
We have two cases.

\vspace{2mm}

\textbf{Case I: }There exists some $i \in [k]$ such that $\sum_{j \in \mc{T}^{c}}|M_{ij}|^{2} \ge 4^{-k}$. In this case, by applying \cref{lem:LKR} with $r = 5^{-k}\sqrt{p}/(3C_{\ref{lem:LKR}})$ and $r_j = |M_{ij}|/3$, we find that
\[\mc{L}\bigg(\sum_{j\in\mc{T}^c}M_{ij}x_j,r\bigg)\le C_{\ref{lem:LKR}}\frac{C_{\ref{lem:LKR}}^{-1}5^{-k}\sqrt{p}/3}{2^{-k}\sqrt{p/9}} < 2^{-k}\le (1-p)^k.\]
Since 
\[\mc{L}(Mx,r)\le\mc{L}((Mx)_i,r)\le\mc{L}\bigg(\sum_{j\in\mc{T}^c}M_{ij}x_j,r\bigg),\]
we have the required conclusion with $\theta = r =  5^{-k}\sqrt{p}/(3C_{\ref{lem:LKR}})$.

\vspace{2mm} 

\textbf{Case II: }For every $i \in [k]$, $\sum_{j \in \mc{T}^{c}}|M_{ij}|^{2} \le 4^{-k}$. Let $U = M_{\mc{T}}$ denote the $k \times |\mc{T}|$ matrix formed by the columns of $M$ corresponding to $\mc{T}$. Note that $k\le |\mc{T}| \le C_{\ref{lem:LKR}}^2 k25^k/p$. Let $C_1,\dots, C_{k}$ denote the columns of $M$. Since $MM^{T} = I$, we have
\begin{align*}
\snorm{I-UU^T} &= \snorm{MM^T-UU^T}\le\sum_{j\in\mc{T}^c}\snorm{C_jC_j^T}\\
&= \sum_{j\in\mc{T}^c}\snorm{C_j}_2^2 = \sum_{i \in [k]}\sum_{j \in \mc{T}^{c}}|M_{ij}|^{2} \\
&\le k4^{-k} < 1/2.
\end{align*}
In particular, the eigenvalues of $UU^T$ are in $(1/2,1]$ so that
\[\on{Tr}[(UU^T)^{-1}] < 2k.\]
Therefore, by \cref{lem:restricted-invertibility}, there exists a set of columns $\mc{S} \subseteq \mc{T}$ of size $|\mc{S}| = k$ such that
\[\snorm{M_{\mc{S}}^{-1}}_{\on{HS}}^2 = \snorm{U_{\mc{S}}^{-1}}_{\on{HS}}^2\le 2C_{\ref{lem:LKR}}^2 k^2 25^k/p.\]
Hence, the smallest singular value of $M_{\mc{S}}$, denoted by $\sigma_k(M_{\mc{S}})$, must satisfy
\[\sigma_k(M_{\mc{S}})\ge\frac{\sqrt{p}}{2C_{\ref{lem:LKR}}k5^k}.\]

We claim that for 
\[r = \frac{\sqrt{p}}{5C_{\ref{lem:LKR}}k5^k},\]
we have that
\[\mc{L}(Mx, r) \le (1-p)^{k}.\]
Since 
\[\mc{L}(Mx, r) \le \mc{L}(M_{\mc{S}}x_{\mc{S}}, r),\]
it suffices to show that
\[\mc{L}(M_{\mc{S}}x_{\mc{S}}, r)\le (1-p)^{k}.\]
Since $x_{\mc{S}}$ is a $k$-dimensional vector with i.i.d.~$\on{Ber}(p)$ entries, it follows that $x_{\mc{S}}$ is supported on $\{0,1\}^k$ with maximum atom probability $(1-p)^k$. Moreover, by definition of the smallest singular value, we see that for any $x \neq y \in \{0,1\}^{k}$,
\[\|M_{\mc{S}}x - M_{\mc{S}}y\|_{2} \ge \sigma_{k}(M_{\mc{S}})\|x-y\|_{2} \ge \sigma_k(M_{\mc{S}}),\]
which shows that
\[\mc{L}(M_{\mc{S}}x_{\mc{S}},r) \le \max_{x \in \{0,1\}^{k}}\mb{P}[x_{\mc{S}} = x] =  (1-p)^k.\]
Thus, in either case, we can take
\[\theta = \frac{\sqrt{p}}{5C_{\ref{lem:LKR}}k5^k}.\qedhere\]
\end{proof}

The previous lemma allows us to quickly deduce the following which, in the special case $k=1$, is the usual `invertibility with respect to a single vector' (cf.~\cite[Lemma~3.5]{Tik20}).

\begin{lemma}\label{lem:fixed-vectors}
Fix $p\in(0,1/2]$, $k\ge 1$, and $\epsilon > 0$. Let $\xi = \on{Ber}(p)$. There exists $c_{\ref{lem:fixed-vectors}} = c_{\ref{lem:fixed-vectors}}(p,k, \epsilon) > 0$ for which the following holds. Let $n \ge n_{\ref{lem:fixed-vectors}}(p,k,\epsilon)$ and let $V = [v_1,\ldots,v_k]$ be an $n\times k$ matrix with orthonormal columns. Then,
\[\mb{P}[\snorm{M_{(n-k)\times n}(\xi)V}_{\on{HS}}\le c_{\ref{lem:fixed-vectors}}\sqrt{n}]\le (1-p+\epsilon)^{kn}.\]
\end{lemma}
\begin{proof}
Denote the rows of $M_{(n-k)\times n}(\xi)$ by $R_1,\ldots,R_{n-k}$. Let $\theta = \theta_{\ref{lem:levy-estimate}}(p,k)$. Note that if \[\snorm{M_{(n-k)\times n}(\xi)V}_{\on{HS}}\le\theta\sqrt{\epsilon'(n-k)},\] then at most $\epsilon'(n-k)$ rows $R_i$ can satisfy $\snorm{R_iV}_2 > \theta$. Denote the set of these rows by $\mc{I}$. By \cref{lem:levy-estimate} and the independence of the rows, we have for all $I\subseteq [n]$, $|I|\le \epsilon'(n-k)$ that
\begin{align*}
    \mb{P}[\mc{I} = I]
    \le \mb{P}[\snorm{R_i V}_{2} \le \theta \hspace{2mm} \forall i \in I^{c}] \le \prod_{i \in I^{c}}\mb{P}[\snorm{V^T(R_i)^T}_{2} \le \theta] \le (1-p)^{k\cdot (1-\epsilon')(n-k)}.
\end{align*}
Therefore, by the union bound over the choice of $I$, we have
\[\mb{P}[\snorm{M_{(n-k)\times n}(\xi)V}_{\on{HS}}\le \theta\sqrt{\epsilon'(n-k)}]\le n \cdot \binom{n-k}{\epsilon'(n-k)}(1-p)^{k\cdot (1-\epsilon')(n-k)}.\]
Therefore taking $c_{\ref{lem:fixed-vectors}} = \theta\sqrt{\epsilon'}/2$ for sufficiently small $\epsilon' = \epsilon'(p, k, \epsilon)$ and taking $n$ sufficiently large gives the desired conclusion.
\end{proof}

Combining this with a standard epsilon-net argument allows us to prove \cref{prop:compressible}.

\begin{proof}[Proof of \cref{prop:compressible}]
Let $\mc{E}(\delta, \rho)$ denote the event that for every $x \in \mb{S}^{n-1}$ such that $$(M_{(n-k)\times n}(\xi))x = 0,$$ we have $x \in \on{Comp}(\delta, \rho)$. By the union bound,
\[\mb{P}[\mc{E}_C(\delta, \rho)] \le \binom{n}{k}\mb{P}[\mc{E}(\delta, \rho)],\]
so that (after rescaling $\epsilon$) it suffices to show that
\[\mb{P}[\mc{E}(\delta, \rho)] \le (1-p+\epsilon)^{kn}.\]

Since the right-kernel of $M_{(n-k)\times n}(\xi)$ has dimension at least $k$, it follows that on the event $\mc{E}(\delta, \rho)$, we can find $k$ orthonormal vectors, $v_1,\dots, v_k$, such that $v_i \in \on{Comp}(\delta, \rho)$ and $(M_{(n-k)\times n}(\xi))v_i = 0$. Let $V$ denote the $n\times k$ matrix with columns $v_1,\dots, v_k$. Then,
\begin{equation}
\label{eqn:certificate}
M_{(n-k)\times n}(\xi)V = 0.
\end{equation}

Let $H \subseteq \mb{R}^{n}$ denote the subspace of vectors $x  = (x_1,\dots, x_n)$ such that $x_1 + \dots + x_n = 0$.
Let $\mc{E}_K$ be the event that the operator norm of $(M_{(n-k)\times n}(\xi))|_{H}$ (i.e.,~the linear operator from $H$ to $\mb{R}^{n-k}$ which coincides with $M_{(n-k)\times n}(\xi)$ on its domain) is at most $K\sqrt{n}$. Since $\xi$ is sub-Gaussian, it follows from standard estimates (cf.~\cite[Lemma~3.4]{Tik20}) that 
\[\mb{P}[\mc{E}_K]\ge 1 - \exp(-c_\xi K^2n).\]
In particular, by choosing $K$ to be of order $\sqrt{k}$, we can ensure that this probability is at least $1-4^{-kn}$. Then, by the union bound, it suffices to show that
\[\mb{P}[\mc{E}(\delta, \rho) \cap \mc{E}_K] \le (1-p+\epsilon)^{kn}.\]

We will show this by combining \cref{lem:fixed-vectors} with a standard epsilon-net argument. Let $\epsilon' > 0$ be a sufficiently small parameter to be chosen later. A standard volumetric net, `densified' in the all-ones direction (to account for the fact that $M_{(n-k)\times n}(\xi)((1,\dots,1)^{T}/\sqrt{n}) = \Theta(n)$) shows (cf.~\cite[Proposition~3.6]{Tik20}) that there exists a (deterministic) net $\mc{N}$ of $\on{Comp}(\delta, \rho)$, of size at most $(C/\epsilon')^{\delta n}$ (where $C$ is allowed to depend on $K$), such that for any $x \in \on{Comp}(\delta, \rho)$, there exists $y \in \mc{N}$ such that on the event $\mc{E}_K$,
\[\|M_{(n-k)\times n}(\xi)(x-y)\|_{2} \le (\epsilon' + \rho)K\sqrt{n}.\]

Let $\mc{V}$ denote the set of all $n\times k$ matrices whose columns are orthonormal vectors in $\on{Comp}(\delta, \rho)$. Then, by considering the $k$-fold product of $\mc{N}$, we obtain a net of $\mc{V}$. Using the standard trick of replacing points in this net by the closest point in $\mc{V}$ (see e.g.~\cite[Lemma~4.2]{Rud14}), we can obtain a (deterministic) net $\mc{M} \subseteq \mc{V}$ of size $|\mc{M}| \le (C/\epsilon')^{\delta kn}$ such that for every $V \in \mc{V}$, there exists $V' \in \mc{M}$ such that, on the event $\mc{E}_K$,
\[\|M_{(n-k)\times n}(\xi)(V-V')\|_{\on{HS}} \le 2(\epsilon'+\rho)K\sqrt{kn}.\]

Therefore, by \cref{eqn:certificate} and the union bound, we have 
\begin{align*}
    \mb{P}[\mc{E}(\delta, \rho) \cap \mc{E}_K] 
    & \le \mb{P}[\mc{E}_K \cap \{\exists V \in \mc{V}: M_{(n-k)\times n}(\xi) V = 0\}]\\
    &\le \sum_{V' \in \mc{M}}\mb{P}[\|M_{(n-k)\times n}(\xi) V'\|_{\on{HS}} \le 2(\epsilon'+\rho)K\sqrt{kn}]\\
    &\le \left(\frac{C}{\epsilon'}\right)^{\delta kn}\cdot (1-p+\epsilon/2)^{kn}\\
    &\le (1-p+\epsilon)^{kn},
\end{align*}
where the penultimate line follows from \cref{lem:fixed-vectors} by taking $\rho = \epsilon'$, $n$ sufficiently large, and $\epsilon'$ sufficiently small depending on $k, p, \epsilon$, and the last line follows by taking $n$ sufficiently large and $\delta$ sufficiently small depending on $\epsilon', \epsilon$. \qedhere


\end{proof}

\subsection{Incompressible vectors}\label{sub:incompressible}
Recall the notion of the threshold of a vector (\cref{def:threshold}). The following is the structure theorem/dichotomy used in the proof of \cref{thm:main}. The case $k = 1$ is implicit in \cite{Tik20}, although the statement given here is closer to the one in work of the authors \cite[Proposition~3.7]{JSS20discrete2}.


\begin{proposition}[Modification of {\cite[Proposition~3.7]{JSS20discrete2}}]
\label{prop:structure}
Let $\delta, \rho, \epsilon \in (0,1)$ and $k\ge 1$. There exist $L_{\ref{prop:structure}} = L_{\ref{prop:structure}}(k,\delta, \rho, p, \epsilon)$ and $n_{\ref{prop:structure}} = n_{\ref{prop:structure}}(k,\delta, \rho, p, \epsilon)$ such that for all $n\ge n_{\ref{prop:structure}}$, with probability at least $1-4^{-kn}$, exactly one of the following holds.
\begin{itemize}
    \item Every unit vector $v$ in the right-kernel of $M_{(n-k)\times n}(\xi)$ is in $\on{Comp}(\delta,\rho)$, or
    \item there is a unit vector $v$ in the right-kernel of $M_{(n-k)\times n}(\xi)$ with $\mc{T}_p(v, L_{\ref{prop:structure}}) \le (1-p+\epsilon)^n$.
\end{itemize}
\end{proposition}
\begin{remark}
The proof of \cref{prop:structure} follows the usual format of taking a dyadic decomposition of possible values for the threshold function, performing randomized rounding on potential kernel vectors at the correct scale, and then tensorizing the resulting small ball probabilities. The difference in the statement above compared to the versions in \cite{Tik20, JSS20discrete1, JSS20discrete2} is that we are missing $k$ rows as opposed to $1$ row. Additionally, we are considering the independent threshold model rather than the ``multislice'' models considered in \cite{JSS20discrete2}, which actually simplifies the proof. Further, since $k$ is sufficiently small compared to $n$, the difference in row counts has essentially no effect on the union bound computation.
\end{remark}

\section*{Acknowledgments}
We thank Jordan Ellenberg for a Twitter question which motivated the present work. We also thank Han Huang for sharing a preliminary version of \cite{Hua21} with us. 

\bibliographystyle{amsplain0.bst}
\bibliography{main.bib}

\end{document}